\documentclass[10pt]{siamonline0516}

\usepackage{mathtools}

\usepackage[latin1]{inputenc}
\usepackage{amssymb}
\usepackage{graphicx}
\usepackage{hyperref}
\usepackage{todonotes}

\newtheorem{Theorem}{Theorem}[section]

\newtheorem{Lemma}{Lemma}[section]

\def \R{\mathbb{R}}

\def \E{\mathbb{E}}

\newcommand{\eps}{\epsilon}

\newcommand{\wrt}{with respect to }

\def \Uc{{\cal U}}

\def\beqs{\begin{eqnarray*}}
\def\enqs{\end{eqnarray*}}
\def\beq{\begin{eqnarray}}
\def\enq{\end{eqnarray}}

\author{Omar Kebiri{\thanks{Laboratory of Statistics and Random Modeling, University of Abou Bekr Belkaid, Tlemcen, Algeria} \footnotemark[2]} \and Lara Neureither\thanks{Institute of Mathematics, Brandenburgische Technische Universität Cottbus-Senftenberg, Cottbus, Germany} \and Carsten Hartmann\footnotemark[2]}

\begin{document}

\title{Singularly perturbed forward-backward stochastic differential equations: application to the optimal control of bilinear systems}

\maketitle
\noindent\textbf{Abstract}:
We study linear-quadratic stochastic optimal control problems with bilinear state dependence for which the underlying stochastic differential equation (SDE) consists of slow and fast degrees of freedom. We show that, in the same way in which the underlying dynamics can be well approximated by a reduced order effective dynamics in the time scale limit (using classical homogenziation results),  the associated optimal expected cost  converges in the time scale limit to an effective optimal cost. This  entails that we can well approximate the stochastic optimal control for the whole system by the reduced order stochastic optimal control, which is clearly easier to solve because of lower dimensionality. The approach uses an equivalent formulation of the Hamilton-Jacobi-Bellman (HJB) equation, in terms of forward-backward SDEs (FBSDEs). We exploit the efficient solvability of FBSDEs via a least squares Monte Carlo algorithm and show its applicability by a suitable numerical example.\\


\section{Introduction}

Stochastic optimal control is one of the important fields in mathematics which has attracted the attention of both pure and applied mathematicians \cite{RS94,FM06}. Stochastic control problems also appear in a variety of applications, such as statistics \cite{WangDupuis2004,is_multiscale}, 
financial mathematics \cite{DN90,P09}, molecular dynamics \cite{SWH12,HS12} or materials science \cite{Steinbrecher2010,Asplund2011}, to mention just a few. 
For some applictions in science and engineering, such as molecular dynamics \cite{SWH12,ZWHWS14}, the high dimensionality of the state space is an important aspect when solving optimal control problems. Another issue when solving optimal control problems by discretising the corresponding dynamic programming equations in space and time are multiscale effects that come into play when the state space dynamics exhibit slow and fast motions.

Here we consider such systems that have slow and fast scales and that are possibly high-dimensional. Several techniques have been developed to reduce the spatial dimension of control systems (see e.g.~\cite{Antbook,Baur2014} and the references therein), but these techniques treat the control as a possibly time-dependent parameter (``open loop control'') and do not take into account that the control may be a feedback control that depends on the state variables (``closed loop control''). 
Clearly, homogenization techniques for  stochastic control systems have been extensively studied by applied 
analysts using a variety of different mathematical tools, including viscosity solutions of the Hamilton-Jacobi-Bellman equation~\cite{Blankenship1987,Evans1989}, backward stochastic differential equations \cite{Buckdahn1998,Ichihara2005}, or occupation measures \cite{kushner1990,Kurtz2001}. The convergence analysis of multiscale stochastic control systems is quite involved and non-constructive, in that the limiting equations of motion are not 
given in explicit or closed form; see \cite{Kokotovic1984,Kabanov2003} for notable exceptions, dealing mainly with the case when the dynamics is linear.

In general, the elimination of variables and solving control problems do not commute, so one of the key questions in control engineering is under which conditions it is possible to eliminate variables before solving an optimal problem. We call this the \emph{model reduction problem}. In this paper we identify a class of stochastic feedback control problems with bilinear state dependence that have the property that they admit the elimination of variables (i.e. model reduction) before solving the control problem. These systems tuern oput to be relevant in the control of high-dimensional transport PDEs, such as Fokker-Planck equations or the evolution equations of open quantum systems \cite{HSZ13,bilinear}. Our approach is based on a Donsker-Varadhan type duality principle between a linear Feynman-Kac PDE and the semi-linear dynamic programming PDE associated with a stochastic control problem \cite{ZLPH14}. Here we exploit the fact that the dynamic programming PDE can be recast as an uncoupled forward backward stochastic differential equation (see e.g. \cite{peng93,to}) that can be treated by model reduction techniques, such as averaging or homogenisation.          

The relation between semilinear PDEs of Hamilton-Jacobi-Bellman type and forward-backward stochastic differential equations (FBSDE) is a classical subject that has been first studied by  Pardoux and Peng \cite{PP90} and since then received lot of attention from various sides, e.g.\cite{BKKM17,DE92,EPQ97,HIM05,HP97,Kobylanski2000}. The  solution theory  has its roots in the work of Antonelli \cite{A93} and since then has been extended in various directions; see e.g. \cite{bgm,BKM,Z05,MPY94}.

From a theoretical point of view, this paper goes beyond our previous works \cite{ZLPH14,HEtal17} in that we prove strong  convergence of the value function and the control without relying on compactness or periodicity assumptions for the fast variables, even though we focus on bilinear systems only, which is the weakest form of nonlinearity. (Many nonlinear systems however can be represented as bilinear systems by a so-called Carleman linearisation.) It also goes beyond the classical works  \cite{Kokotovic1984,Kabanov2003}  that treat systems that are either fully linear or linear in the fast variables. We stress that we are mainly aiming at the model reduction problem, but we discuss alongside with the theoretical results some ideas to discretise the corresponding FBSDE \cite{B97,C97,BET09,BenderSteiner2012,HNM17}, since one of the main motivations for doing model reduction is to reduce the \emph{numerical} complexity of solving optimal control problems.   

\subsection{Set-up and problem statement}

We briefly discuss the technical set-up of the control problem considered in this paper.  
In this paper, we consider the linear-quadratic (LQ) stochastic control problem of the following form: minimize the expected cost
\begin{equation}\label{jofu}
J(u;t,x) = \E\!\left[\int_t^\tau  \left(q_0(X_s^{u}) + |u_s|^2\right)ds + q_1(X_\tau^{u})\,\Big|\,X^{u}_t=x\right]
\end{equation}
over all admissible controls $u\in\Uc$ and subject to:
\begin{equation}\label{sde}
dX_s^{u} = \left(a(X^{u}_s) + b(X_s^{u})u_s\right)ds + \sigma(X_s^u)dW_s\,,\quad 0\leqslant t\leqslant s\leqslant \tau\,.
\end{equation}
Here $\tau<\infty$ is a bounded stopping time (specified below), and the set of admissible controls $\Uc$ is chosen such that (\ref{sde}) has a unique strong solution. The denomination \emph{linear-quadratic} for (\ref{jofu})--(\ref{sde}) is due to the specific dependence of the system on the control variable $u$.
The state vector $x\in\R^n$ is assumed to be high-dimensional, which is why we seek a low-dimensional approximation of (\ref{jofu})--(\ref{sde}).

Specifically, we consider the case that $q_0$ and $q_1$ are quadratic in $x$, $a$ is linear and $\sigma$ is constant, and the control term is an affine function of $x$, i.e.,
\[
b(x)u = \left(N x + B\right)u
\]
In this case the system is called \emph{bilinear} (including linear systems as a special case), and the aim is to replace (\ref{sde}) by a lower dimensional bilinear system
\begin{equation*}
d\bar{X}_s^{v} = \bar{A}\bar{X}^{v}_s\,ds + \left(\bar{N} \bar{X}_s^{v} + \bar{B}\right)v_{s}\,ds + \bar{C}dW_s\,,\quad 0\leqslant t\leqslant s\leqslant \tau\,,
\end{equation*}
with states $\bar{x}\in\R^{n_s}$, $n_s\ll n$ and an associated reduced cost functional
\begin{equation*}
\bar{J}(v;\bar{x},t) = \E\!\left[\int_t^\tau \left(\bar{q}_0(\bar{X}_s^{v}) + |v_s|^2\right)ds + \bar{q}_1(\bar{X}_\tau^{v})\,\Big|\,\bar{X}^{v}_t=\bar{x}\right]\,,
\end{equation*}
that is solved instead of (\ref{jofu})--(\ref{sde}). Letting $v^*$ denote the minimizer of $\bar{J}$, we require that $v^*$ is a good approximation of the minimizer $u^*$ of the original problem where ''good approximation'' is understood in the sense that
\begin{equation*}
J(v^*;\cdot,t=0) \approx  J(u^*;\cdot,t=0)\,.
\end{equation*}
In the last equation, closeness must be suitably interpreted, e.g.~uniformly on all compact subsets of $\R^n\times[0,T)$ for some $T<\infty$.

One situation in which the above approximation property holds is when  $u^* \approx v^*$ uniformly in $t$ and the cost is continuous in the control, but it turns out that this requirement will be too strong in general and overly restrictive. We will discuss alternative criteria in the course of this paper.

\subsection{Outline}

The paper is organised as follows: In Section \ref{sec:soc} we introduce the bilinear stochastic control problem studied in this paper and  derive the corresponding forward-backward stochastic differential equation (FBSDE). Section \ref{sec:mor} contains the main result, a convergence result for the value function of a singularly perturbed control problem with bilinear state dependence, based on an FBSDE formulation. In Section 4 we present a numerical example to illustrate the theoretical findings and discuss the numerical discretization of the FBSDE. The article concludes in Section \ref{sec:out} with a short summary and a discussion of future work. The proof of the main result and some technical lemmas are recorded in the Appendix.

\section{Singularly perturbed bilinear control systems}\label{sec:soc}

We now specify the system dynamics (\ref{sde}) and the corresponding cost functional (\ref{jofu}). Let $(x_1,x_2)\in\R^{n_s}\times\R^{n_f}$ with $n_s+n_f=n$ denote a decomposition of the state vector $x\in\R^n$ into relevant (slow) and irrelevant (fast) components. Further let $W=(W_t)_{t\ge 0}$ denote $\mathbb{R}^m$-valued Brownian motion on a probability space $(\Omega,\mathcal{F},P)$ that is endowed with the filtration $(\mathcal{F}_t)_{t\ge 0}$ generated by $W$. For any initial condition $x\in\mathbb{R}^{n}$ and any $\mathcal{A}$-valued admissible control $u\in\Uc$, with $\mathcal{A}\subset\R$, we consider the following system of It\^o stochastic differential equations
\begin{equation}
\label{controbideq}
dX^{\epsilon}_s = A X^{\epsilon}_s\,ds + (N X^{\epsilon}_s + B)u_{s}\, ds + C dW_s\,,\; X^{\epsilon}_t = x\,,
\end{equation}
that depends parametrically on a parameter $\eps>0$ via the coefficients
\[A=A^\eps\in\R^{n\times n}\,,\; N=N^\eps\in\R^{n\times n}\,,\; B=B^\eps\in \R^{n}\,,\textrm{ and } C=C^\eps\in\R^{n\times m}\,,
\] where for brevity we also drop the dependence of the process on the control $u$, i.e. $X^\eps_s=X^{u,\eps}_s$.
The stiffness matrix $A$ in (\ref{controbideq}) is assumed to be of the form
\begin{equation}\label{coeffA}
A= \left(\begin{array}{@{}c@{\quad}c}
\begin{array}{ccc}
		A_{11}
		\end{array} & \begin{array}{ccc}
		{\epsilon}^{-1/2}     A_{12}
		\end{array}\\ \\
	{\epsilon}^{-1/2}   A_{21} & {\epsilon}^{-1} A_{22}
	\end{array}\right) \in \R^{(n_s+n_f)\times (n_s+n_f)}\,,
\end{equation}
with $n=n_s+n_f$. Control and noise coefficients are given by
\begin{equation}\label{coeffN}
N= \left(\begin{array}{@{}c@{\quad}c}
\begin{array}{ccc}
N_{11}
\end{array} & \begin{array}{ccc}
  N_{12}
\end{array} \\ \\
{\epsilon}^{-1/2}   N_{21} & {\epsilon}^{-1/2} N_{22}
\end{array}\right)\in \R^{(n_s+n_f)\times (n_s+n_f)}
\end{equation}
and
\begin{equation}\label{coeffBC}
B=\left(
\begin{array}{cc}
B_1 \\
{\epsilon}^{-1/2} B_2 \\
\end{array}
\right)\in\R^{(n_s+n_f)\times 1},\quad C=\left(
\begin{array}{cc}
C_1 \\
{\epsilon}^{-1/2} C_2 \\
\end{array}
\right)\in\R^{(n_s+n_f)\times m}\,,
\end{equation}
where $Nx+B\in{\rm range}(C)$ for all $x\in\R^n$; often we will consider either the case $m=1$ with $C_i=\sqrt{\rho}B_i$, $\rho>0$, or $m=n$, with $C$ being a multiple of the identity when $\eps=1$.
All block matrices $A_{ij}, N_{ij}$, $B_i$ and $C_j$ are assumed to be order 1 and independent of $\eps$.

The above $\eps$-scaling of coefficients is natural for a system with $n_s$ slow and $n_f$ fast degrees of freedom and arises, for example, as a result of a balancing transformation applied to a large-scale system of equations; see e.g.~\cite{H11,HSZ13}.
A special case of (\ref{controbideq}) is the linear system
\begin{equation}
\label{controledeq}
dX^{\epsilon}_s = \left(A X^{\epsilon}_s + B u_{s}\right) ds +  C dW_s\,.
\end{equation}

Our goal is to control the stochastic dynamics (\ref{controbideq})---or (\ref{controledeq}) as a special variant---so that a given cost criterion is optimised. Specifically, given two symmetric positive semidefinite matrices  $Q_0,Q_1\in\R^{n_s\times n_s}$, we consider the quadratic cost functional
\begin{equation}
\label{bertfuctional}
J(u;t,x) = \E\left[\frac{1}{2} \int_t^\tau ((X_{1,s}^{\epsilon})^\top  Q^{}_0 X_{1,s}^{\epsilon}+|u_s|^2)ds+\frac{1}{2}(X_{1,\tau}^{\epsilon})^\top  Q^{}_1 X_{1,\tau}^{\epsilon}\right],
\end{equation}
that we seek to minimize subject to the dynamics (\ref{controbideq}). Here the expectation is understood over all realisations of $(X_s^\eps)_{s\in[t,\tau]}$ starting at $X_t^\eps=x$, and as a consequence $J$ is a function of the initial data $(t,x)$. The stopping time is defined as the minimum of some time $T<\infty$ and the first exit time of a domain $D=D_s\times \R^{n_f}\subset\R^{n_s}\times\R^{n_f}$ where $D_s$ is an open and bounded set with smooth boundary. Specifically, we set $\tau=\min\{\tau_D,T\}$, with
\[
\tau_D = \inf\{s\ge t : X^\eps_s \notin D\}\,.
\]
In other words, $\tau$ is the stopping time that is defined by the event that either $s=T$ or $X^\eps_{s}$ leaves the set $D=D_s\times\R^{n_f}$, whichever comes first.
Note that the cost function does not explicitly depend on the fast variables $x_2$. We define the corresponding value function by
\begin{equation}\label{value}
V^\eps(t,x) = \inf_{u\in\Uc} J(u;t,x)\,.
\end{equation}

\subsubsection*{Remark}
As a consequence of the boundedness of $D_s\subset\R^{n_s}$, we may assume that all coefficients in our control problem are bounded or Lipschitz continuous, which makes some of the proofs in the paper more transparent. 

We further note that all of the following considerations trivially carry over to the case $N=0$ and a multi-dimensional control variable, i.e., $u\in\R^k$ and $B\in\R^{n\times k}$.

\subsection{From stochastic control to forward-backward stochastic differential equations}

We suppose that the matrix pair $(A,C)$ satisfies the Kalman rank condition
\begin{equation}\label{controllablility}
{\rm rank}(C|AC|A^2C|\ldots|A^{n-1}C) = n\,.
\end{equation}
A necessary---and in this case sufficient---condition for optimality of our optimal control problem is that the value function (\ref{value}) solves a semilinear parabolic partial differential equation of Hamilton-Jacobi-Bellman type (a.k.a.~\emph{dynamic programming equation}) \cite{fleming2005}
\begin{equation}\label{hjb}
- \frac{\partial V^\eps}{\partial t} = L^\eps V^\eps + f(x,V^\eps,C^\top \nabla V^\eps)\,, \quad V^\eps|_{E^{+}}= q_1\,,
\end{equation}
where
\[
q_1(x) = \frac{1}{2}x_1^\top Q_1x_1^{}\,
\]
and $E^{+}$
is the terminal set of the augmented process $(s,X_{s}^\eps)$, precisely $E^{+} = \left([0,T)\times\partial D\right)\cup\left(\{T\}\times D\right)$.
Here $L^\eps$ is the infinitesimal generator of the control-free process,
\begin{equation}\label{L}
L^\eps = \frac{1}{2}CC^\top \colon\nabla^2 + (A x)\cdot\nabla\,,
\end{equation}
and the nonlinearity $f$ is independent of $\eps$ and given by
\begin{equation}\label{driver}
f(x,y,z) = \frac{1}{2}x_1^\top  Q_0 x_1^{} - \frac{1}{2} \big|\left(x^\top  N^\top  + B^\top \right)\left(C^\top \right)^\sharp z\big|^2\,.
\end{equation}
Note that $f$ is furthermore independent of $y$ and that the Moore-Penrose pseudoinverse
\[
\left(C^\top \right)^\sharp=C(C^\top C)^{-1}
\]
is unambiguously defined since $z=C^\top \nabla V^\eps$ and  $(Nx+B)\in{\rm range}(C)$, which by noting that  $\left(C^\top \right)^\sharp C^\top $ is the orthogonal projection onto ${\rm range}(C)$ implies that
\[
\big|(x^\top N^\top +B^\top )\nabla V^\eps\big|^2 = \big|(x^\top  N^\top  + B^\top )\left(C^\top \right)^\sharp z\big|^2\,.
\]

The specific semilinear form of the equation is a consequence of the control problem being linear-quadratic. As a consequence, the dynamic programming equation (\ref{hjb}) admits a representation in form of an uncoupled forward-backward stochastic differential equation (FBSDE). To appreciate this point, consider the control-free process $X_s^\eps=X_s^{\eps,u=0}$ with infinitesimal generator $L^\eps$ and define an adapted process $Y_s^\eps=Y_s^{\eps,x,t}$ by
\begin{equation}\label{Ydef}
Y_s^\eps = V^\eps(s,X_s^\eps)\,.
\end{equation}
(We abuse notation and denote both the controlled and the uncontrolled process by $X_s^\eps$.)
Then, by definition, $Y_t^\eps = V^\eps(x,t)$. Moreover, by It\^o's formula and the dynamic programming equation (\ref{hjb}), the pair $(X_s^\eps,Y_s^\eps)_{s\in[t,\tau]}$ can be shown to solve the system of equations
\begin{equation}\label{fbsde}
\begin{aligned}
dX_s^\eps & = AX_s^\eps\,ds + C\,dW_s\,,\quad X^\eps_t=x\\
dY_s^\eps & = - f(X_s^\eps,Y_s^\eps,Z_s^\eps)ds + Z_s^\eps\,dW_s\,,\quad Y_\tau^\eps = q_1(X_\tau^\eps)\,,
\end{aligned}	
\end{equation}
with $Z_s^\eps = C^\top \nabla V^\eps(s,X_s^\eps)$ being the control variable. Here, the second equation is only meaningful if interpreted as a backward equation, since only in this case $Z_s^\eps$ is uniquely defined. To see this, let $f=0$ and $q_{1}(x)=x$ and note that the ansatz (\ref{Ydef}) implies that $Y_{s}^{\eps}$ is adapted to the filtration generated by the forward process $X_{s}^{\eps}$. If the second equation was just a time-reversed SDE then $(Y_{s}^{\eps},Z^{\eps}_{s})\equiv (X_{\tau}^{\eps},0)$ would be the unique solution to the SDE $dY_s^\eps = Z_s^\eps\,dW_s$ with terminal condition $Y_\tau^\eps = X_\tau^\eps$. But such a solution would not be adapted, because $Y_{s}^{\eps}$ for $s<\tau$ would  depend on  the future value $X_{\tau}^{\eps}$ of the forward process.

\subsection*{Remark}
Equation (\ref{fbsde}) is called an \emph{uncoupled FBSDE} because the forward equation for $X_s^\eps$ is independent of $Y_s^\eps$ or $Z_s^\eps$. The fact that the FBSDE is uncoupled furnishes a well-known duality relation between the value function of an LQ optimal control problem and the cumulate generating function of the cost \cite{BD00,DMR96}; specifically, in the case that $N=0$, $B=C$ and the pair $(A,B)$ being completely controllable, it holds that
\begin{equation}\label{duality}
V^\eps(x,t) = -\log\E\left[\exp\left(-\int_t^\tau q_0(X_s^{\epsilon})ds-q_1(X_\tau^{\epsilon})\right)\right],
\end{equation}
with
\[
q_0(x) = \frac{1}{2}x_1^\top  Q_0 x_1^{}\,.
\]
Here the expectation on the right hand side is taken over all realisations of the control-free process $X_s^\eps=X_s^{\eps,u=0}$, starting at $X_t^\eps=x$. By the Feynman-Kac theorem, the function $\psi^\eps=\exp(-V^\eps)$ solves the linear parabolic boundary value problem
\begin{equation}\label{fk}
\left(\frac{\partial}{\partial t} + L^\eps\right)\psi^\eps = q_0(x)\psi^\eps\,, \quad \psi^\eps|_{E^{+}}= \exp\left(-q_1\right)\,,
\end{equation}
which is equivalent to the corresponding dynamic programming equation (\ref{hjb}). 

\section{Model reduction}\label{sec:mor}

The idea now is to exploit the fact that (\ref{fbsde}) is uncoupled, which allows us to derive an FBSDE for the slow variables $\bar{X}_s^\eps=X_{1,s}^\eps$ only, by standard singular perturbation methods. The reduced FBSDE as $\eps\to 0$ will then be of the form
\begin{equation}\label{rfbsde}
\begin{aligned}
d\bar{X}_s & = \bar{A}\bar{X}_s\,ds + \bar{C}\,dW_s\,,\quad \bar{X}_t=x_1\\
d\bar{Y}_s & = - \bar{f}(\bar{X}_s,\bar{Y}_s,\bar{Z}_s)ds + \bar{Z}_s\,dW_s\,,\quad \bar{Y}_\tau = \bar{q}_1(\bar{X}_\tau)\,,
\end{aligned}	
\end{equation}
where the limiting form of the backward SDE follows from the corresponding properties of the forward SDE. Specifically, assuming that the solution of the associated SDE
\begin{equation}\label{assSDE}
	d\xi_s = A_{22}\xi_s ds + C_2 dW_s\,,
\end{equation}
that is governing the fast dynamics as $\eps\to 0$, is ergodic with unique Gaussian invariant measure $\pi={\mathcal N}(0,\Sigma)$, where $\Sigma=\Sigma^\top >0$ is the unique solution to the Lyapunov equation
\begin{equation}\label{lyap}
A_{22}\Sigma + \Sigma A_{22}^\top  = -C_2C_2^\top \,,
\end{equation}
we obtain that, asymptotically as $\eps\to 0$,
\begin{equation}\label{fastProc}
	X_{2,s}^{\eps} \sim \xi_{s/\eps}\,,\quad s>0\,.
\end{equation}
As a consequence, the limiting SDE governing the evolution of the slow process $X_{1,s}^\eps$--- in other words: the forward part of (\ref{rfbsde})---has the coefficients
\begin{equation}
\label{limitCoeff1}
\bar{A}=A^{}_{11}-A^{}_{12}A^{-1}_{22}A^{}_{21}\,,\quad  \bar{C}=C_1-A^{}_{12}A^{-1}_{22}C^{}_{2}\,,
\end{equation}
as following from standard homogenisation arguments \cite{PS08}; a formal derivation is given in the appendix. By a similar reasoning we find that the driver of the limiting backward SDE reads
\begin{equation}\label{limitDriver1}
\bar{f}(x_1,y,z_1) = \int\limits_{\R^{n_f}}  f((x_1,x_2),y,(z_1,0))\,\pi(dx_2)\,,
\end{equation}
specifically,
\begin{equation}\label{limitDriver2}
\bar{f}(x_1,y,z_1) = \frac{1}{2} x_1^\top  \bar{Q}^{}_0 x^{}_1 - \frac{1}{2}\big|\left(x_1^\top  \bar{N}^\top _{} + \bar{B}^\top _{}\right) z^{}_1\big|^2 + K^{}_0\,,
\end{equation}
with
\begin{equation}\label{limitCoeff2}
\bar{Q}^{}_0  = Q_0\,,\quad \bar{N} = C_1^\sharp N^{}_{11}\,,\quad \bar{B} = C_1^\sharp\left(B^{}_1 +  N^{}_{12}\Sigma^{1/2}_{}\right)\,.
\end{equation}
The limiting backward SDE is equipped with a terminal condition $\bar{q}_1$ that  equals $q_1$, namely,
\begin{equation}
\bar{q}_1(x_1) = \frac{1}{2}x_1^\top  Q_{1}x^{}_1\,.
\end{equation}

\subsection*{Interpretation as an optimal control problem}

It is possible to interpret the reduced FBSDE again as the probabilistic version of a dynamic programming equation. To this end, note that (\ref{controllablility}) implies that the matrix pair $(\bar{A},\bar{C})$ satisfies the Kalman rank condition \cite{LuiAnderson1989}
\[
{\rm rank}(\bar{C}|A\bar{C}|A^2\bar{C}|\ldots|A^{n_s-1}\bar{C}) = n_s\,.
\]
As a consequence, the semilinear partial differential equation
\begin{equation}\label{rhjb}
- \frac{\partial V}{\partial t} = \bar{L}V  + \bar{f}(x_1,V,\bar{C}^\top \nabla V)\,, \quad V|_{E_s^{+}}= \bar{q}_1\,,
\end{equation}
with $E_s^{+} = \left([0,T)\times\partial D_s\right)\cup\left(\{T\}\times D_s\right)$ and
\begin{equation}\label{Lbar}
\bar{L} = \frac{1}{2}\bar{C}\bar{C}^\top \colon\nabla^2 + (\bar{A} x_1)\cdot\nabla\,
\end{equation}
has a classical solution $V\in C^{1,2}([0,T)\times D)\cap C^{0,1}(E_s^{+})$. Letting $\bar{Y}_s:=V(s,\bar{X}_s)$, $0\leqslant t\leqslant s\leqslant \tau$, with initial data $\bar{X}_t=x_1$ and $\bar{Z}_s=\bar{C}^\top \nabla V(s,\bar{X}_s)$, the limiting FBSDE (\ref{rfbsde}) can be readily seen to be equivalent to
(\ref{rhjb}). The latter is the dynamic programming equation of the following LQ optimal control problem: minimize the cost functional
\begin{equation}\label{jofv}
\bar{J}(v;t,x_1) = \E\left[\frac{1}{2} \int_t^\tau  (\bar{X}_s^\top  \bar{Q}^{}_0 \bar{X}^{}_s + |v_s|^2)ds + \frac{1}{2}\bar{X}_{\tau}^\top  \bar{Q}^{}_1 \bar{X}^{}_{\tau}\right],
\end{equation}
subject to
\begin{equation}\label{rsde}
d\bar{X}_s = \bar{A} \bar{X}_s ds + \left(\bar{M}\bar{X}_s + \bar{D}\right) v_s \,ds + \bar{C} dw_s\,,\quad  \bar{X}_t = x_1\,,
\end{equation}
where $(w_s)_{s\ge 0}$ denotes standard Brownian motion in $\R^{n_s}$ and we have introduced the new control coefficients $\bar{M}=\bar{C}\bar{N}$ and $\bar{D}=\bar{C}\bar{B}$.

\subsection{Convergence of the control value}

Before we state our main result and discuss its implications for the model reduction of linear and bilinear systems, we recall that basic assumptions that we impose on the system dynamics. Specifically, we say that the dynamics (\ref{controbideq}) and the corresponding cost functional (\ref{bertfuctional}) satisfy \textbf{Condition LQ} if the following holds:

\begin{enumerate}
	\item $(A,C)$ is controllable, and the range of $b(x)=Nx+B$ is a subspace of ${\rm range}(C)$.\label{assCC}
	\item The matrix $A_{22}$ is Hurwitz (i.e., its spectrum lies entirely in the open left complex half-plane) and the matrix pair $(A_{22},C_2)$ is controllable.\label{assErgode}
	\item The driver of the FBSDE (\ref{fbsde}) is continuous and quadratically growing in $Z$. \label{assQuad}
	\item The terminal condition in (\ref{fbsde}) is bounded; for simplicity we set $Q_1=0$ in (\ref{bertfuctional}).\label{assBdd}
\end{enumerate}

Assumption \ref{assErgode} implies that the fast subsystem (\ref{assSDE}) has a unique Gaussian invariant measure $\pi=\mathcal{N}(0,\Sigma)$ with full topological support, i.e., we have $\Sigma=\Sigma^\top >0$.
According to \cite[Prop.~3.1]{BBM86} and \cite{Kobylanski2000}, existence and uniqueness of (\ref{fbsde}) is guaranteed by Assumptions \ref{assQuad} and \ref{assBdd} and the controllability of $(A,C)$ and the range condition, which imply that the transition probability densities of the (controlled or uncontrolled) forward process $X_s^\eps$ are smooth and strictly positive. As a consequence of the complete controllability of the original system, the reduced system (\ref{rsde}) is completely controllable too, which guarantees existence and uniqueness of a classical solution of the limiting dynamic programming equation (\ref{rhjb}); see, e.g., \cite{PP92}.

Uniform convergence of the value function $V^\eps\to V$ is now entailed by the strong convergence of the solution to the corresponding FBSDE as is expressed by the following Theorem.

\begin{Theorem}\label{thm:value}
Let the assumptions of Condition LQ hold. Further let $V^\eps$ be the classical solution of the dynamic programming equation (\ref{hjb}) and $V$ be the solution of (\ref{rhjb}). Then
\[
V^\eps \to  V\,,
\]
uniformly on all compact subsets of $[0,T]\times D$.
\end{Theorem}
The proof of the Theorem is given in Appendix \ref{sec:value}. For the reader's convenience, we present a formal derivation of the limit equation in the next subsection.

\subsection{Formal derivation of the limiting FBSDE}\label{sec:derivation}

Our derivation of the limit FBSDE follows standard homogenisation arguments (see \cite{FW12,K63,PS08}), taking advantage of the fact that the FBSDE is uncoupled. To this end we consider the following linear evolution equation
\begin{equation}\label{bke}
	\left(\frac{\partial}{\partial t} - L^\eps\right) \phi^\eps = 0\,,\quad \phi^\eps(x_1,x_2,0) = g(x_1)\,
\end{equation}
for a function $\phi^\eps\colon  \bar{D}_s\times\R^{n_f}\times[0,T]$ where
\begin{equation}\label{Leps}
L^\eps = \frac{1}{\eps} L_0 + \frac{1}{\sqrt{\eps}}L_1 + L_2\,,
\end{equation}
with
\begin{subequations}\label{Li}
\begin{alignat}{3}
L_0 & = \frac{1}{2}C_{2}C_{2}^\top \colon\nabla_{x_2}^2 + (A_{22} x_2)\cdot\nabla_{x_2}\\
L_1 & = \frac{1}{2}C_1C_2^\top \colon\nabla^2_{x_2x_1}+ \frac{1}{2}C_2C_1^\top \colon\nabla^2_{x_1x_2}+ (A_{12} x_2)\cdot\nabla_{x_1}+(A_{21} x_1)\cdot\nabla_{x_2}\\
L_2 & = \frac{1}{2}C_{1}C_{1}^\top \colon\nabla_{x_1}^2 + (A_{11} x_1)\cdot\nabla_{x_1}
\end{alignat}
\end{subequations}
is the generator associated with the control-free forward process $X^\eps_s$ in (\ref{fbsde}). We follow the standard procedure of \cite{PS08} and consider the perturbative expansion
\[
\phi^\eps = \phi_0 + \sqrt{\eps}\phi_{1} + \eps\phi_{2} + \ldots
\]
that we insert into the Kolmogorov equation (\ref{bke}). Equating different powers of $\eps$ we find a hierarchy of equations, the first three of which read
\begin{equation}\label{epsComp}
	L_0\phi_0 = 0\,,\quad L_0\phi_{1} = - L_1\phi_{0}\,,\quad	L_0\phi_2 = \frac{\partial\phi_0}{\partial t} - L_1\phi_1 - L_2\phi_0\,.
\end{equation}
Assumption \ref{assErgode} on page \pageref{assErgode} implies that $L_0$ has a one-dimensional nullspace that is spanned by functions that are constant in $x_2$, and thus the first of the three equations implies that $\phi_0$ is independent of $x_1$. Hence the second equation---the cell problem---reads
\begin{equation}\label{cell}
L_0\phi_1 = -(A_{12}x_2)\cdot\nabla\phi_0(x_1,t)\,.
\end{equation}
The last equation has a solution by the Fredholm alternative, since the right hand side averages to zero under the invariant measure $\pi$ of the fast dynamics that is generated by the operator $L_0$, in other words, the right hand side of the linear equation is orthogonal to the nullspace of $L_0^*$ spanned by the density of $\pi$.\footnote{Here $L_0^*$ is the formal $L^2$ adjoint of the operator $L_0$, defined on a suitable dense subspace of $L^2$.} The form of the equation suggests the general ansatz
\[\phi_1 = \psi(x_2)\cdot\nabla\phi_0(x_1,t) + R(x_1,t)
\]
where the function $R$ plays no role in what follows, so we set it equal to zero. Since $L_0\psi = -(A_{12}x_2)^\top $, the function $\psi$ must be of the form $\psi=Qx_2$ with a matrix $Q\in\R^{n_s\times n_f}$. Hence
\[
Q = -A_{12}A_{22}^{-1}\,.
\]
Now, solvability of the last of the three equations requires again that the right hand side averages to zero under $\pi$, i.e.
\begin{equation}\label{homogEqn0}
\int_{\R^{n_f}}\left(\frac{\partial\phi}{\partial t} + L_1\left[\left(A_{12}A_{22}^{-1}x_2\right)\cdot\nabla\phi\right] - L_2\phi\right)\pi(dx_2)\,,
\end{equation}
which formally yields the limiting equation for $\phi=\phi_0(x_1,t)$. Since $\pi$ is a Gaussian measure with  mean $0$ and covariance $\Sigma$ given by (\ref{lyap}), the integral (\ref{homogEqn0}) can be explicitly computed:
\begin{equation}\label{homogEqn1}
\left(\frac{\partial}{\partial t} - \bar{L}\right)\phi\ =0,,\quad \phi(x_1,0) = g(x_1)\,,
\end{equation}
where $\bar{L}$ is given by (\ref{Lbar}) and the initial condition $\phi(\cdot,0)=g$ is a consequence of the fact that the initial condition in (\ref{bke}) is independent of $\eps$. By the controllability of the pair $(\bar{A},\bar{C})$, the limiting equation (\ref{homogEqn1}) has a unique classical solution and uniform convergence $\phi^\eps\to\phi$ is guaranteed by standard results, e.g., \cite[Thm.~20.1]{PS08}.

Since the backward part of (\ref{fbsde}) is independent of $\epsilon$, the final form of the homogenised FBSDE (\ref{rfbsde}) is found by averaging over $x_2$, with the unique solution of the corresponding backward SDE satisfying $Z_{2,s}=0$ as the averaged backward process is independent of $x_2$.

\section{Numerical studies}\label{sec:num}
In this section we presents numerical results for linear and bilinear control systems and discuss the numerical discretisation of uncoupled FBSDE associated with LQ stochastic control problems. We begin with the latter.

\subsection{Numerical FBSDE discretisation}

The fact that (\ref{fbsde}) or (\ref{rfbsde}) are decoupled entails that they can be discretised by an explicit time-stepping algorithm. Here we utilize a variant of the least-squares Monte Carlo algorithm proposed in \cite{BenderSteiner2012}; see also \cite{GobetEtal2017}. The convergence of numerical schemes for FBSDE with quadratic nonlinearities in the driver has been analysed in \cite{TurkedjievDissertation2013}.

The least-squares Monte Carlo scheme is based on the Euler discretisation of (\ref{fbsde}):
\begin{equation}\label{fbsdeEuler}
	\begin{aligned}
	\hat{X}_{n+1}  & = \hat{X}_n + \Delta t A\hat{X}_n + \sqrt{\Delta t}C\xi_{n+1}\\
	\hat{Y}_{n+1} & = \hat{Y}_{n} - \Delta t f(\hat{X}_n,\hat{Y}_n,\hat{Z}_n) + \sqrt{\Delta t}\hat{Z}_n\cdot\xi_{n+1}
	\end{aligned}
\end{equation}
where $(\hat{X}_n,\hat{Y}_n)$ denotes the numerical discretisation of the joint process $(X^\eps_s,Y^\eps_s)$, where we set $X_s^\eps=X_{\tau_D}^\eps$ for $s\in(\tau_D,T]$ when $\tau_D<T$, and $(\xi_k)_{k\ge 1}$ is an i.i.d.~sequence of normalised Gaussian random variables.
Now let
\[
\mathcal{F}_n = \sigma\big(\big\{\hat{W}_k: 0\leqslant k\leqslant n\big\}\big)
\]
be the $\sigma$-algebra generated by the discrete Brownian motion $\hat{W}_n:=\sqrt{\Delta t}\sum_{i\leqslant n}\xi_i$. By definition the joint process $(X_s^\eps,Y_s^\eps)$ is adapted to the filtration generated by $(W_r)_{0\leqslant r\leqslant s}$, therefore
\begin{equation}\label{condExp}
\hat{Y}_n = \E\big[\hat{Y}_n|\mathcal{F}_n\big] = \E\big[\hat{Y}_{n+1} + \Delta t f(\hat{X}_n,\hat{Y}_n,\hat{Z}_n)|\mathcal{F}_n\big]\,,
\end{equation}
where we have used that $\hat{Z}_n$ is independent of $\xi_{n+1}$. In order to compute $\hat{Y}_n$ from $\hat{Y}_{n+1}$ we use the identification of $Z^\eps_s$ with $C^\top \nabla V^\eps(s,X^\eps_s)$ and replace $\hat{Z}_n$ in (\ref{condExp}) by
\begin{equation}\label{Zn}
\hat{Z}_n = C^\top \nabla V^\eps (t_n,\hat{X}_n)\,,
\end{equation}
which, the parametric ansatz (\ref{YofX}) for the $V^\eps$ makes the overall scheme explicit in $\hat{X}_n$ and $\hat{Y}_n$.

\subsubsection*{Least-squares solution of the backward SDE}
In order to evaluate the conditional expectation $\hat{Y}_n=\E[\cdot|\mathcal{F}_n]$ we recall that a conditional expectation can be characterised as the solution to the following quadratic minimisation problem:
\[
\E\big[S|\mathcal{F}_n\big] = \mathop{\rm argmin}_{Y\in L^2,\, \mathcal{F}_n \textrm{-measurable}}\E[|Y-S|^2]\,.
\]
Given $M$ independent realisations $\hat{X}_n^{(i)}$, $i=1,\ldots,M$ of the forward process $\hat{X}_n$, this suggests the approximation scheme
\begin{equation}\label{condVar}
\hat{Y}_n \approx \mathop{\rm argmin}_{Y=Y(\hat{X}_n)}\frac{1}{M}\sum_{i=1}^{M}\Big|Y -  \hat{Y}_{n+1}^{(i)} - \Delta t f\big(\hat{X}^{(i)}_n,\hat{Y}^{(i)}_{n+1},C^\top \hat{Y}^{(i)}_{n+1}\big)\Big|^2\,,
\end{equation}
where $\hat{Y}^{(i)}$ is defined by $\hat{Y}^{(i)}=Y\big(\hat{X}^{(i)}\big)$ with terminal values
\[
\hat{Y}^{(i)}_N = q^{}_1\big(X^{(i)}_N\big)\,\quad \tau=N\Delta t\,.
\]
(Note that $N=N_D$ is random.)
For simplicity, we assume in what follows that the terminal value is zero, i.e., we set $q_1=0$. (Recall that the existence and uniqueness result from \cite{Kobylanski2000} requires $q_1$ to be bounded.) To represent $\hat{Y}_n$ as a function $Y(\hat{X}_n)$ we use the ansatz
\begin{equation}\label{YofX}
Y(\hat{X}_n) = \sum_{k=1}^{K} \alpha_k(n) \varphi_k(\hat{X}_n)\,,
\end{equation}
with coefficients $\alpha_1(\cdot),\ldots,\alpha_K(\cdot)\in\R$
and suitable basis functions $\varphi_1,\ldots,\varphi_K\colon\R^n\to\R$ (e.g.~Gaussians). Note that the coefficients $\alpha_k$ are the unknowns in the least-squares problem (\ref{condVar}) and thus are independent of the realisation.
Now the least-squares problem that has to be solved in the $n$-th step of the backward iteration is of the form
\begin{equation}\label{leastSq}
\hat{\alpha}(n) = \mathop{\rm argmin}_{\alpha\in\R^K} \left\|A_n\alpha - b_n\right\|^2\,,
\end{equation}
with coefficients
\begin{equation}\label{leastSqA}
A_n = \left(\varphi_k\Big(\hat{X}_n^{(i)}\Big)\right)_{i=1,\ldots,M;k=1,\ldots,K}\,
\end{equation}
and data
\begin{equation}\label{leastSqb}
b_n = \left(\hat{Y}_{n+1}^{(i)} - \Delta t f\big(\hat{X}^{(i)}_n,\hat{Y}^{(i)}_{n+1},C^\top \hat{Y}^{(i)}_{n+1}\big)\right)_{i=1,\ldots,M}\,.
\end{equation}
Assuming that the coefficient matrix $A_n\in\R^{M\times K}$, $K\leqslant M$ defined by (\ref{leastSqA}) has maximum rank $K$, then the solution to the least-squares problem (\ref{leastSq}) is given by
\begin{equation}\label{leastSqSol}
	\hat{\alpha}(n) = \left(A_n^\top A^{}_n\right)^{-1}A_n^\top  b^{}_n\,.
\end{equation}

The thus defined scheme is strongly convergent of order 1/2 as $\Delta t\to 0$ and $M,K\to\infty$ as has been analysed by \cite{BenderSteiner2012}. Controlling the approximation quality for finite values $\Delta t, M, K$, however, requires a careful adjustment of the simulation parameters and appropriate basis functions, especially with regard to the condition number of the matrix $A_n$.
\subsection{Numerical example}
Illustrating our theoretical findings of Theorem \ref{thm:value}, we consider a linear system  of form \eqref{controledeq}
where the matrices $A,\ B$ and $C$ are given by
\begin{equation*}
A= \left(\begin{array}{@{}c@{\quad}c}
\begin{array}{ccc}
        0
        \end{array} & \begin{array}{ccc}
        {\epsilon}^{-1/2}  \,   I_{n \times n}
        \end{array}\\ \\
    -{\epsilon}^{-1/2}  \,   I_{n \times n} & -\gamma\, {\epsilon}^{-1}\, I_{n \times n}
    \end{array}\right) \in \R^{2n\times 2n}\,,
\end{equation*}
and
\begin{equation*}
B=C=\left(
\begin{array}{cc}
0 \\
\sigma \,{\epsilon}^{-1/2} \, I_{n \times n} \\
\end{array}
\right)\in\R^{2n\times n}\,.
\end{equation*}

This is an instance of a controlled Langevin equation with friction and noise coeffcient $\gamma,\sigma > 0$ which are assumed to fulfill the fluctation-dissipation relation 
\[2 \gamma = \sigma^2\,.\] 
In the example we let $\gamma = 1/2$ and $\sigma = 1$. The quadratic cost functional \eqref{bertfuctional} is determined by the running cost via $Q_0= I_{n \times n} \, \in \mathbb{R}^{n \times n}
$ and we apply no terminal cost, i.e. $ Q_1= 0. $\\
The associated effective equations are given by \eqref{jofv}--\eqref{rsde}, where \[
 \bar{A} = - \gamma^{-1} I_{n \times n},\quad  \bar{D}=\bar{C} = \sigma \gamma^{-1},\quad, \bar{M}=0,\quad \bar{Q}_0 = I_{n \times n},\quad \bar{Q}_1 = 0 \quad \in \mathbb{R}^{n \times n}.
\]
We apply the previously described FBSDE scheme \eqref{fbsdeEuler},\eqref{YofX},\eqref{leastSq}--\eqref{leastSqSol}, which was shown to yield good results in \cite{IHP}, to both the full and the reduced system, and we choose $n=3$, i.e the full system is six dimensional. To this end we choose the basis functions
\[\phi^{\mu_k, \delta}_{k,n}(x) = \exp\left(-\frac{(\mu_k-x)^2}{2 \delta}\right) \] where $\delta = 0.1$ is fixed but $\mu_k = \mu_k(n)$ changes in each timestep such that the basis follows the forward process. For this, we simulate $K$ additional forward trajectories $X^{(k)}, k=1,\ldots,K $ and set $\mu_k(n) = X^{(k)}_n$. 

We choose the parameters for the numerics as follows. The number of basis functions $K$ is given by $K=9$ for the reduced system and $K^{\epsilon}=40$ for the full system. We choose these values because the maximally observed rank of the matrices $A_n$ defined in \eqref{leastSqA} is 9 for the reduced system and we want these matrices to have rank $K$. For the full system we could have used a greater values for $K$, but we want to keep the computational effort reasonable. Further, we choose $\Delta t = 5 \cdot 10^{-5}$ , the final time $T = 0.5$ and the number of realisations $M=400.$ \\
We let the whole algorithm run five times and compute the distance between the value functions of the full and reduced systems \[E(\epsilon):=|V^{\epsilon}(0,x) - V(0,x)|\] for which convergence of order $1/2$ was found in the proof of Theorem \ref{thm:value}. Indeed, this is the order of convergence which we also observe in the numerics of our example as can be seen in figure \ref{fig:convorder} where we depict the mean and standard deviation of $E(\epsilon)$.
\begin{figure}[h]
 \centering
  \includegraphics[width=0.7\textwidth]{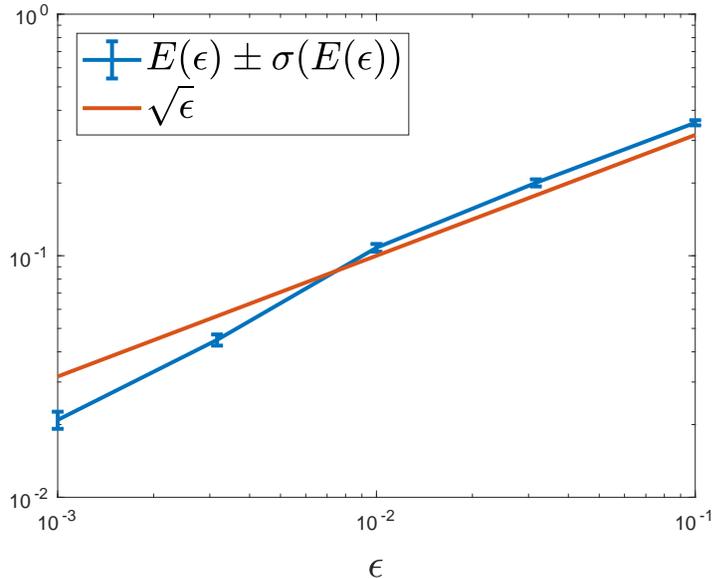}
  \caption{Plot of the mean of $E(\epsilon) \ \pm $ its standard deviation $(\sigma(E(\epsilon)))$ and for comparison of $\sqrt{\epsilon}$ against $\epsilon$ on a doubly logarithmic scale: we observe convergence of order $1/2$ as predicted by the theory. } \label{fig:convorder}

\end{figure}

\subsection{Discussion}

We shall now discuss the implications of the above simple example when it comes to more complicated dynamical systems. As a general remark the results show that it is possible to to apply model reduction before solving the corresponding optimal control problem where the control variable in the original equation can simply be treated as a parameter. This is in accordance with the general model reduction strategy in control engineering; see e.g. \cite{Antbook,Baur2014}  and the references therein. Our results not only guarantee convergence of the value function via convergence of $Y^{\eps}$, but they also imply strong convergence of the optimal control, by the convergence of the control process $Z^{\eps}$ in $L^{2}$. (See the appendix for details.)        
This means that in the case of a system with time scale separation, our result is highly valuable since we can resort to the reduced system for finding the optimal control which can then be applied to the full systems dynamics. %

We stress that our results carry over to fully nonlinear stochastic control problems which have a similar LQ structure \cite{ZLPH14}. Clearly, for realistic (i.e. high-dimensional or nonlinear) systems the identification of a small parameter $\eps$ remains challenging, and one has to resort to e.g.~semi-empirical approaches, such as \cite{Franzke05,Lall2002}. 

If the dynamics is linear, as is the case here, small parameters may be identified using system theoretic arguments based on balancing transformations (see, e.g., \cite{H11,HSZ13}). These approaches require that the dynamics is either linear or bilinear in the state variables, but the aforementioned duality for the quasi-linear dynamic programming equation can be used here as well in order to change the drift of the forward SDE from some nonlinear vector field, say, $b$ to a linear vector field $b_0=Ax$. Assuming that the noise coefficient $C$ is square and invertible and ignoring $\eps$ and the boundary condition for the moment, it is easy to see that the dynamic programming PDE (\ref{hjb}) can be recast as
\begin{align*}
-\frac{\partial V^{\eps}}{\partial t} = \tilde{L}V + \tilde{f}(x,V^{\eps},C^{\top}\nabla_{x}V^{\eps}) = 0\,,
\end{align*}
Here 
\[
\tilde{L} = \frac{1}{2}CC^{\top} + b(x)\cdot\nabla \]
is the generator of a forward SDE with nonlinear drift $b$, and
\[
\tilde{f}(x,y,z) = f(x,y,z) +  C^{-1}(Ax - b(x))\cdot z\,.
\]
is the driver of the corresponding backward SDE. Even though the change of drift is somewhat arbitrary, it shows that by changing the driver in the backward SDE it is possible to reduce the control problem to one with linear drift that falls within the category that is considered in this paper, at the expense of having a possibly non-quadratic cost functional. 

\subsubsection*{Remark}
Changing the drift may be advantageous in connection with the numerical FBSDE solver. In the martingale basis approach of Bender and Steiner \cite{BenderSteiner2012}, the authors have suggested to use basis functions that are defined as conditional expectations of certain linearly independent candidate functions over the forward process, which makes the basis functions martingales. Computing the martingale basis, however, comes with a large computational overhead, which is why the authors consider only cases in which the conditional expectations can be computed analytically. Changing the drift of the forward SDE may thus be used to simplify the forward dynamics so that its distribution becomes analytically tractable.

\section{Conclusions and outlook}\label{sec:out}
We have given a proof of concept that model reduction methods for singularly perturbed bilinear control systems can be applied to the dynamics, before solving the corresponding optimal control problem. The key idea that to connect the HJB corresponding to our stochastic optimal control which is a semi-linear PDE to a perturbed forward-backward SDE which is decoupled, so  we benefit from this end to derive a reduced FBSDE as the perturbation parameter $\eps$ goes to zero. The reduced FBSDE can then be interpreted as a reduced stochastic control problem, and we have proved uniform convergence of the corresponding value function. As an auxiliary  result, we obtain that the optimal control converges as well in a strong sense, which implies that the optimal control computed from the reduced system can be used to control the original dynamics. 

We presented numerical results for linear control system and we discussed the numerical discretisation of uncoupled FBSDE, based on the computation of conditional expectations. For the latter the choice of the basis functions plays an essential role, and how to cleverly choose the ansatz functions, possibly exploiting that the forward SDE has an explicit solution (see e.g. \cite{BenderSteiner2012}) is an important topic, especially for high dimensional problems. We leave the question regarding the adaptive choice of ansatz functions to future work.  

Another class of important problems, that we have not considered in this article, are slow-fast systems with vanishing noise. The natural question here is how the limit equation depend on the order in which noise and time scale parameters go to zero. This question has important consequences for the associated deterministic control problems and its regularisation by noise. We leave this topic for future work too.  

\appendix

\section{Proofs and technical lemmas}\label{sec:proofs}

The idea of the proof of Theorem \ref{thm:value} closely follows the work \cite{BH99}, with the main differences being (a) that we consider slow-fast systems exhibiting three time scales, in particular the slow equation contains singular $\mathcal{O}(\eps^{-1/2})$ terms, and (b) that the coefficients of the fast dynamics are not periodic, with the fast process being asymptotically Gaussian as $\eps\to 0$; in particular the $n_f$-dimensional fast process lives on the unbounded domain $\R^{n_f}$.

\subsection{Poisson equation Lemma}\label{sec:poisson}

Theorem \ref{thm:value} rests on the following Lemma that is similar to a result in \cite{BLP78}.

\begin{Lemma}
	\label{lemadriver}
	Suppose that the assumptions of Condition LQ on page \pageref{assBdd} hold and define $h\colon[0,T]\times \R^{n_s}\times \mathbb{R}^{n_f} \to \mathbb{R}$ to be a function of the class $C_b^{1,2,2}$. Further assume that  $h$ is centered with respect to the invariant measure $\pi$ of the fast process.
	Then for every $t\in[0,T]$ and initial conditions  $(X_{1,u}^\eps,X_{2,u}^\eps) = (x_1,x_2) \in \R^{n_s}\times\mathbb{R}^{n_f}$, $0\leqslant u< t$,  we have
	\begin{equation}
	\lim_{\eps\to 0}\E\left[ \left( \int_{u}^{v} h(s,X^{\epsilon}_{1,s},X^{\epsilon}_{2,s})ds \right)^2 \right] = 0\,,\quad 0\leqslant u < v\leqslant t\,.
	\end{equation}
\end{Lemma}

\begin{proof}
	We remind the reader of the definition (\ref{Li}) of the differential operators $L_0,\,L_1$ and $L_2$, and consider the Poisson equation
	\begin{equation}
	\label{PDElemma}
	L_0 \psi=-h
	\end{equation}
	on the domain $\R^{n_f}$. (The variables $x_1\in \R^{n_s}$ and $t\in[0,T]$ are considered as parameters.) Since $h$ is centered \wrt $\pi$, equation (\ref{PDElemma}) has a solution by the Fredholm alternative. 	
	By Assumption \ref{assErgode} $L_0$ is a hypoelliptic operator in $x_2$ and thus by \cite[Thm.~2]{Veretennikov3}, the Poisson equation (\ref{PDElemma}) has a unique solution that is smooth and bounded.
	Applying It{\^o}'s formula to $\psi$ and introducing the shorthand
	$\delta\psi(u,v) = \psi(v,X^{\epsilon}_{1,v},X_{2,v}^{\epsilon})- \psi(u,x_{1},x_{2})$
		yields
	\begin{equation}\label{deltaPsi}
	\begin{aligned}
	\delta\psi(u,v) = & \int_{u}^{v}(\partial_t\psi+L_2 \psi)(s,X^{\epsilon}_{1,s},X_{2,s}^{\epsilon})ds +  \frac{1}{\sqrt{\epsilon}}\int_{u}^{v}L_1 \psi(s,X^{\epsilon}_{1,s},X_{2,s}^{\epsilon})ds\\
	& + \frac{1}{\epsilon}\int_{u}^{v} L_0 \psi(s,X^{\epsilon}_{1,s},X_{2,s}^{\epsilon})ds + M_1(u,v) + \frac{1}{\sqrt{\eps}}M_2(u,v)\,,
	\end{aligned}
	\end{equation}
	where $M_1$ and $M_2$ are square integrable martingales \wrt the natural filtration generated by the Brownian motion $W_s$:
		\begin{equation}\label{M1M2}
	\begin{aligned}
	M_1(u,v) = & \int_{u}^{v}(\partial_t\psi+L_2 \psi)(s,X^{\epsilon}_{1,s},X_{2,s}^{\epsilon})ds +  \frac{1}{\sqrt{\epsilon}}\int_{u}^{v}L_1 \psi(s,X^{\epsilon}_{1,s},X_{2,s}^{\epsilon})ds\\
	& + \frac{1}{\epsilon}\int_{u}^{v} L_0 \psi(s,X^{\epsilon}_{1,s},X_{2,s}^{\epsilon})ds + M_1(u,v) + \frac{1}{\sqrt{\eps}}M_2(u,v)\,,
	\end{aligned}
	\end{equation}

	 By the properties of the solution to (\ref{PDElemma}) the first three integrals on the right hand side are uniformly bounded in $u$ and $v$, and thus
	\begin{align*}
	\int_{u}^{v} h(s,X^{\epsilon}_{1,s},X_{2,s}^{\epsilon})ds  = & - \eps\delta\psi(u,v) + \eps\int_{u}^{v}(\partial_t\psi+L_2 \psi)(s,X^{\epsilon}_{1,s},X_{2,s}^{\epsilon})ds \\ & +  \sqrt{\epsilon}\int_{u}^{v}L_1 \psi(s,X^{\epsilon}_{1,s},X_{2,s}^{\epsilon})ds+ \eps M_1(u,v) + \sqrt{\eps}M_2(u,v)\,.
	\end{align*}
		
	By the It\^{o} isometry and the boundedness of the derivatives $\nabla_{x_1}\psi$ and $\nabla_{x_2}\psi$, the martingale term can be bounded by
	\[
	\E\left[(M_i(u,v))^2\right] \leqslant C_i(v-u)\,, \quad 0<C_i<\infty\,.
	\]
	Hence
	\[
	\E\left[\left(\int_{u}^{v} h(s,X^{\epsilon}_{1,s},X_{2,s}^{\epsilon})ds\right)^2\right] \leqslant C\eps\,,
	\]
	with a generic constant $0<C<\infty$ that is independent of $u,\,v$ and $\eps$.
\end{proof}

\subsection{Convergence of the value function}\label{sec:value}

\begin{Lemma}
	Suppose that Condition LQ from page \pageref{assCC} holds. Then
	\[
	|V^{\epsilon}(t,x)-V(t,x_1)|\leq C \sqrt{\epsilon}\,,
	\]
	with $x=(x_1,x_2)\in D=D_s\times\R^{n_f}$,
	where $V^{\epsilon}$ is the solution of the original dynamic programming equation (\ref{hjb}) and $V$ is the solution of the limiting dynamic programming equation (\ref{rhjb}). The constant and $C$ depends on $x$ and $t$, but is finite on every compact subset of $D\times[0,T]$.
\end{Lemma}

\begin{proof}
	The idea of the proof is to apply It{\^o}'s formula to  $|y^{\epsilon}_s|^2$, where $y^{\epsilon}_s=Y^{\epsilon}_s-V(s,X_{1,s}^{\epsilon})$ satisfies the backward SDE
	\begin{equation}
	\label{BSDEG}
	dy^{\epsilon}_s=-G^{\epsilon}(s,X_{1,s}^\epsilon,X_{2,s}^\epsilon, y^{\epsilon}_s, z^{\epsilon}_s ) ds +z^{\epsilon}_s \cdot dW_s
	\end{equation}
	where
	 \[
	 z^{\epsilon}_s=Z^{\epsilon}_s-\left(\bar{C}^\top \nabla V(s,X_{1,s}^\eps),\,0\right)^\top  \qquad (\nabla V = \nabla_{x_1}V)
	 \]
	 and
	 \[
	 G^{\epsilon}(t,x_1,x_2,y,z)  = G_1(t,x_1,x_2,y,z) + G^{\epsilon}_2(t,x_1,x_2,y,z)\,,
	 \]
	 with
	 \begin{align*}
	 	G_1  & = f(t,x,y+V(t,x_1),z+ (\bar{C}^\top \nabla V(t,x_1),0)) -\bar{f}(t,x_1,V(t,x_1), \bar{C}^\top  \nabla V(t,x_1))\\
		G_2^\eps & =   \left((A_{11}-\overline{A})x_1+\frac{1}{\eps}A_{12}x_2\right)\cdot\nabla V(t,x_1) +\frac{1}{2} (C_1^{}C_1^\top - \bar{C}\bar{C}^\top)\nabla^2 V(t,x_1)\,.
	 \end{align*}
	We set $X_s^\eps=X_{\tau_D}^\eps$ for $s\in(\tau_D,T]$ when $\tau_D<T$. Then, by construction, $G_1(t,x,0,0)$, $x=(x_1,x_2)\in D_s\times\R^{n_f}$ is centered \wrt $\pi$ and bounded (since the running cost is independent of $x_2$), therefore Lemma \ref{lemadriver} implies that
	\begin{equation}\label{G1}	\sup_{t\in[0,T]}\E\left[\left(\int_t^T G_1(s,X_{1,s}^\epsilon,X_{2,s}^\epsilon,0,0)ds\right)^2\right] \leqslant C_1 \epsilon\,,
	\end{equation}
	The second contribution to the driver can be recast as
	$G_2^\eps = (L-\bar{L})V$, with $L_2$ and $\bar{L}$ as given by (\ref{L}) and (\ref{Lbar}) and thus, as $\eps\to 0$, 	
	\begin{equation}\label{G2}
	\sup_{t\in[0,T]}\E\left[\left(\int_t^T G_2^\eps(s,X_{1,s}^\epsilon,X_{2,s}^\epsilon,0,0) ds \right)^2\right] \leqslant C_2 \epsilon
	\end{equation}
	by the functional central limit theorem for diffusions with Lipschitz coefficients \cite{FW12}; cf.~also Sec.~\ref{sec:derivation}. As a consequence of (\ref{G1}) and (\ref{G2}), we have $G^\eps\to 0$ in $L^2$, which, since $\E[|y^{\epsilon}_T|^2] \leqslant C_3 \epsilon$, implies strong convergence of the solution of the corresponding backward SDE in $L^2$.
	
	Specifically, since $\nabla V$ is bounded $\bar{D}_s$, It{\^o}'s formula applied to $|y^{\epsilon}_s|^2$, yields after an application of Gronwall's Lemma:
	\begin{align*}
	\E\left[\sup_{t\le s\le T} |y^{\epsilon}_s|^2+\int_t^{T} |z^{\epsilon}_s|^2\,ds\right]\leqslant & \ell_{D} \E\left[\left(\int_t^{T} G^\eps(s,X_{1,s}^\epsilon,X_{2,s}^\epsilon,0,0) ds\right)^2\right] + \ell_{D} \E[|y^{\epsilon}_{T}|^2]
	\end{align*}
	where the Lipschitz constant $\ell_D$ is independent of $\epsilon$ and finite for every compact subset $\bar{D}_s\subset\R^{n_s}$ by the boundedness of $\nabla V$ (since $V$ is a classical solution and $D_s$ in bounded). Hence $\E[|y^\eps_s|^2]\le C_3\eps$ uniformly for $s\in[t,T]$, and by setting $s=t$, we obtain
	\[
	|Y_t^\eps| = |V^\eps(t,x) - V(t,x_1)| \le C\sqrt{\eps}
	\]
	for a constant $C\in(0,\infty)$.
\end{proof}
This proves Theorem \ref{thm:value}.

\subsubsection*{Acknowledgements}
This research has been partially funded by Deutsche Forschungsgemeinschaft (DFG) through the grant CRC 1114 "Scaling Cascades
in Complex Systems", Project A05 "Probing scales in equilibrated systems by
optimal nonequilibrium forcing". Omar Kebiri acknowledges funding from the EU-METALIC II Programme.


\begin{thebibliography}{99}

\bibitem{LuiAnderson1989} B. D. O. Anderson and Y. Liu. Controller reduction: concepts and approaches. IEEE Trans. Autom. Control 34, 802--812 (1989).
\bibitem{A93} F. Antonelli, \textit{Backward-forward stochastic differential equations}. Ann. Appl. Probab. 3 (1993),
no. 3, 777-793.


\bibitem{Antbook}
A.C. Antoulas, {\em Approximation of large-scale dynamical
  systems}, Advances in design and control, Society for Industrial and Applied
  Mathematics (2005).


\bibitem{Asplund2011} E. Asplund and T. Kl\"uner, {\em Optimal control of open quantum systems applied to the photochemistry
  of surfaces}, Phys. Rev. Lett. 106, 140404 (2011).



\bibitem {bgm} K. Bahlali , B. Gherbal , B. Mezerdi, \textit{Existence of optimal controls for systems driven by FBSDEs}, Syst. Control Letters 60, 344-349 (1995).

\bibitem{BKKM17} K. Bahlali, O. Kebiri, N. Khelfallah and H. Moussaoui \textit{One dimensional BSDEs with logarithmic growth application to PDEs}. Stochastics, 1744-2516 (2017).

\bibitem {BKM} K. Bahlali, O. Kebiri, A. Mtiraoui: \textit{Existence of an optimal Control for a system driven by a degenerate coupled Forward-Backward Stochastic Differential Equations}, C. R. Acad. Sci. Paris, Ser. I (2016).

\bibitem{B97} V. Bally Approximation scheme for solutions of BSDE. In: El Karoui, N., Mazliak, L. (eds.) Backward Stochastic Differential Equations, Addison Wesley Longman (1997), 177-191.

\bibitem {BenderSteiner2012} C. Bender, J. Steiner: \textit{Least-Squares Monte Carlo for BSDEs}. In: Carmona et al. (Eds.), Numerical Methods in Finance, Springer, (2012) 257-289.


\bibitem{Baur2014}
U. Baur, P. Benner, and L. Feng, {\em Model order reduction for
  linear and nonlinear systems: A system-theoretic perspective}, Arch. Comput.
  Meth. Eng. 21, 331--358 (2014).


\bibitem{BBM86} A. Bensoussan, L. Boccardo, F. Murat, \textit{Homogenization of elliptic equations with principal part not in divergence form and hamiltonian with quadratic growth}, Commun. Pure Appl. Math. 39 (1986) 769-805.

\bibitem{BLP78} A. Bensoussan, J. L. Lions, G. Papanicolaou, \textit{Asymptotic Analysis for Periodic Structures}, North-Holland,Amsterdam, (1978) 769-805.

\bibitem{Blankenship1987} A. Bensoussan and G. Blankenship, {\em Singular perturbations in stochastic control},
 in: \emph{Singular Perturbations and Asymptotic Analysis in Control
  Systems} (eds. P. V. Kokotovic, A. Bensoussan, and G. L. Blankenship), vol.~90
  of Lecture Notes in Control and Information Sciences, Springer Berlin Heidelberg, pp.~171--260 (1987).


\bibitem{BKT05} B. Bouchard, N. E. Karoui and N. Touzi, \textit{Maturity randomization for stochastic control problems}, Ann. Appl. Probab. (2005), Vol. 15, No. 4, 2575-2605
 \bibitem{BET09} B. Bouchard, R. Elie, N. Touzi, \textit{Discrete-time approximation of BSDEs and probabilistic schemes for fully nonlinear PDEs}, Advanced financial modelling, Radon Ser. Comput. Appl. Math., 8, Walter de Gruyter, Berlin, (2009) 91-124.
\bibitem {BH99} P. Briand, Y. Hu, \textit{Probabilistic approach to singular perturbations of semilinear and quasilinear parabolic} , Nonlinear Analysis 35, 815-831 (1999). 


\bibitem{Buckdahn1998} R. Buckdahn and Y. Hu, {\em Probabilistic approach to homogenizations of systems of quasilinear
  parabolic {PDEs} with periodic structures}, Nonlinear Analysis 32, 609 -- 619 (1998).


\bibitem {BD00} A. Budhiraja and P. Dupuis, \textit{A variational representation for positive functionals of infinite dimensional Brownian motion}, Probab. Math. Statist., 20 (2000), pp. 39-61.
\bibitem{C97} D. Chevance, Numerical methods for backward stochastic differential equations, Numerical methods in finance, Publ. Newton Inst., Cambridge Univ. Press, Cambridge (1997) 232-244.

\bibitem {DMR96} P. Dai Pra., L. Meneghini, and W. J. Runggaldier, \textit{Connections between stochastic control and dynamic games}, Math. Control Signals Systems, 9 (1996), pp. 303-326.

\bibitem{DN90} M. H. Davis and A. R. Norman, \textit{Portfolio selection with transaction costs}, Math. Oper. Res., 15 (1990), pp. 676-713
\bibitem{DE92} D. Duffie and L. G. Epstein, \textit{Stochastic differential utility}. Econometrica, 60(2), (1992)



\bibitem{is_multiscale}
P. Dupuis, K. Spiliopoulos and H. Wang, {\em Importance sampling for multiscale diffusions}, Multiscale Model. Simul. 10, 1--27, (2012).


\bibitem{EPQ97} N. El Karoui, S. Peng, and M. C. Quenez. \textit{Backward stochastic differential equations in finance.}  Mathematical Finance, 1, 1-71 (1997).


\bibitem{Evans1989} L. C. Evans, {\em The perturbed test function method for viscosity solutions of
  nonlinear {PDE}}, P. Roy. Soc. Edinb. A 111, 359--375 (1989).


\bibitem{fleming2005} W. H. Fleming, \textit{Optimal investment models with minimum consumption criteria}, Australian Economic Papers 44, 307-321 (2005).

\bibitem {FM06} W. H. Fleming and H. Mete Soner \textit{Controlled Markov processes and viscosity solutions}. Applications of mathematics. Springer, New York, 2nd edition, (2006).


\bibitem{Franzke05} C. Franzke and A.J. Majda, and E. Vanden-Eijnden, {\em  Low-order stochastic mode
reduction for a realistic barotropic model climate.} J. Atmos. Sci. 62, 1722--1745 (2005).

\bibitem {FW12} M. Freidlin and A. Wentzell, \textit{Random Perturbations of Dynamical Systems}, vol. 260 of Grundlehren der mathematischen Wissenschaften, Springer Berlin Heidelberg, 2012.

\bibitem{GobetEtal2017}
E.~Gobet and P.~Turkedjiev, \textit{Adaptive importance sampling in least-squares Monte Carlo algorithms for backward stochastic differential equations}, Stoch. Proc. Appl. 127, 1171-1203 (2005). 



\bibitem{H11} C. Hartmann, \textit{Balanced model reduction of partially-observed Langevin equations}: an averaging principle, Math. Comput. Model. Dyn. Syst. 17, 463-490, (2011). 

\bibitem{ZLPH14}  C. Hartmann, J. Latorre, G. A. Pavliotis, and  W. Zhang, \textit{Optimal control of multiscale systems using reduced-order models,} J. Computational Dynamics 1, 279-306 (2014)

\bibitem{HSZ13} C. Hartmann, B. Schäfer-Bung and A. Zueva \textit{Balanced averaging of bilinear systems with applications to stochastic control} SIAM J. Control Optim. 51, 2356-2378, (2013). 

\bibitem {HS12}  C. Hartmann and C. Sch\"{u}tte, \textit{Efficient rare event simulation by optimal nonequilibrium forcing}, J. Stat. Mech. Theor. Exp. 2012, P11004 (2012).

\bibitem{HEtal17} 
 C. Hartmann, C. Sch\"utte, M. Weber, and W. Zhang {\em Importance sampling in path space for diffusion processes with slow-fast variables,} Probab. Theory Relat. Fields 170, 177--228 (2017).

    \bibitem{HIM05} Y. Hu, P. Imkeller, and M. M\"uller. {\em Utility maximization in incomplete markets}, Ann. Appl. Probab. 15, 1691--1712 (2005).

\bibitem {HP97} Y. Hu, S. Peng, \textit{A stability theorem of backward stochastic differential equations and its application}, C. R. Acad. Sci. Paris, Ser. I Math. 324, 1059--1064 (1997.
   \bibitem{HNM17} C. B. Hyndman, P. O. Ngou, \textit{A Convolution Method for Numerical Solution of Backward Stochastic Differential Equations} Methodol. Comput. Appl. Probab. 19,   1--29 (2017).
   
   
\bibitem{Ichihara2005} N. Ichihara, {\em A stochastic representation for fully nonlinear {PDE}s and its
  application to homogenization}, J. Math. Sci. Univ. Tokyo 12, 467--492 (2005).

\bibitem{IHP} O. Kebiri, L. Neureither, and C. Hartmann, \textit{Adaptive importance sampling with forward-backward stochastic differential equations}, Submitted (2018) . 

\bibitem {K63} R. Khasminskii, \textit{Principle of averaging for parabolic and elliptic differential equations and for Markov processes with small diffusion}, Theory Probab. Appl. 8, 1-21 (1963).


\bibitem{Kabanov2003} Y. Kabanov and S. Pergamenshchikov, \emph{Two-scale stochastic systems: asymptotic analysis and control}, Springer, Berlin, Heidelberg, Paris (2003).

\bibitem{Kokotovic1984} P. V. Kokotovic, {\em Applications of singular perturbation techniques to control problems}, SIAM Review 26, 501--550 (1984).



\bibitem {Kobylanski2000} M. Kobylanski, {\em Backward stochastic differential equations and partial differential equations with quadratic growth}, Ann. Probab. 28, 558--602 (2000).


\bibitem{kushner1990} H. J. Kushner, \emph{Weak Convergence Methods and Singularly Perturbed Stochastic
  Control and Filtering Problems}, Birkh\"auser, Boston (1990).
  
  \bibitem{Kurtz2001} T. Kurtz and R. H. Stockbridge, {\em Stationary solutions and forward equations for controlled and
  singular martingale problems}, Electron. J. Probab 6, 5 (2001).

	
\bibitem{MPY94} J. Ma, P. Protter, and J. Yong. \textit{Solving Forward-Backward Stochastic Differential Equations Explicitly-a Four Step Scheme} Probability Theory and Related Fields, 98, 339-359 (1994).
	
	
\bibitem{Lall2002}
S. Lall, J. Marsden, and S. Glava\v{s}ki, {\em A subspace approach to balanced truncation for model reduction of
  nonlinear control systems}, Int. J. Robust Nonlinear Control 12, 519--535 (2002).

\bibitem{bilinear} P. M. Pardalos and V. A. Yatsenko, {\em Optimization and Control of Bilinear Systems: Theory, Algorithms, and Applications}, Springer US, 2010. 


\bibitem{PP92} E. Pardoux, S. Peng, \textit{Backward stochastic differential equations and quasilinear parabolic partial differential equations}, in: B.L. Rozovskii, R.B. Sowers (Eds.), Stochastic Partial Differential Equations
and their Applications, Lecture Notes in Control and Information Sciences 176, Springer, Berlin, (1992).

\bibitem{Veretennikov3} E. Pardoux and A. Yu. Veretennikov: \textit{On the poisson equation and diffusion approximation 3}, The Annals of Probability, Vol. 33, No. 3, 1111--1133 (2005).
\bibitem{PP90} E. Pardoux and S. Peng. \textit{Adapted solution of a backward stochastic differential equation}. System Control Letters, 14(1):55-61, (1990).
\bibitem {PS08} G. A. Pavliotis and A. M. Stuart,\textit{ Multiscale Methods: Averaging and Homogenization}, Springer, (2008).

\bibitem {peng93} S. Peng, \textit{Backward Stochastic Differential Equations and Applications to Optimal Control}, Appl. Math. Optim. 27 (1993), pp.125-144.

\bibitem{P09} H. Pham, \textit{Continuous-time stochastic control and optimization with financial applications}, Stochastic modelling and applied probability, Springer, Berlin, Heidelberg, (2009).

\bibitem{Steinbrecher2010} A. Steinbrecher, {\em Optimal control of robot guided laser material treatment}, in: \emph{Progress in Industrial Mathematics at ECMI 2008} (eds. A.~D.
  Fitt, J. Norbury, H. Ockendon, and E. Wilson),Springer Berlin Heidelberg. pp., 501--511 (2010). 
  

\bibitem {RS94} F. Robert  Stengel. \textit{Optimal control and estimation. Dover books on advanced mathematics}. Dover Publications, New York, (1994).

\bibitem {SWH12}  C. Sch\"{u}tte, S. Winkelmann, and C. Hartmann  , \textit{Optimal control of molecular dynamics using markov state models}, Math. Program. Ser. B 134, 259-282 (2012).

\bibitem {to} N. Touzi \textit{Optimal stochastic control, stochastic target problem, and backward differential equation}, Springer-Verlag (2013).

\bibitem{TurkedjievDissertation2013} P. Turkedjiev: \textit{Numerical methods for backward stochastic differential equations of quadratic and locally Lipschitz type}, Dissertation, Humboldt-Universität zu Berlin, Mathematisch-Naturwissenschaftliche Fakultät II (2013).


\bibitem{WangDupuis2004}
P. Dupuis and H. Wang, {\em Importance sampling, large deviations, and differential games}, Stochastics and Stochastic Reports 76, 481--508 (2004),

\bibitem{ZWHWS14} W. Zhang, H. Wang, C. Hartmann, M. Weber and C. Sch\"{u}tte, \textit{Applications of the cross-entropy method to importance sampling and optimal control of diffusions}, SIAM J. Sci. Comput. 36, A2654-A2672, (2014)

\bibitem {Z05} W. Zhen, \textit{Forward-backward stochastic differential equations, linear quadratic stochastic optimal control and nonzero sum differential games}, Journal of Systems Science and Complexity (2005)





\end{thebibliography}
\end{document}